\documentclass[11pt]{amsart}
\usepackage{amssymb, amstext, amscd, amsmath, amssymb}
\usepackage{mathtools, xypic, paralist, color, dsfont, rotating}
\usepackage{verbatim}
\usepackage{enumerate}

\numberwithin{equation}{section}

\renewenvironment{quote}
  {\list{}{\rightmargin=.5cm \leftmargin=.5cm}%
   \item\relax}
  {\endlist}

\usepackage[noadjust]{cite}
\usepackage{filecontents}
 
\makeatletter
\def\@cite#1#2{{\m@th\upshape\bfseries%
[{#1\if@tempswa{\m@th\upshape\mdseries, #2}\fi}]}}
\makeatother
\theoremstyle{plain}
\newtheorem{theorem}{Theorem}[section]
\newtheorem{corollary}[theorem]{Corollary}
\newtheorem{proposition}[theorem]{Proposition}
\newtheorem{lemma}[theorem]{Lemma}
\theoremstyle{definition}
\newtheorem{definition}[theorem]{Definition}

\newtheorem{remark}[theorem]{Remark}

\newtheorem*{acknow}{Acknowledgements}
\theoremstyle{remark}

\renewcommand{\qedsymbol}{{\vrule height5pt width5pt depth1pt}}
\mathtoolsset{centercolon}
  \newcommand{\A}{{\mathcal{A}}}
  \newcommand{\B}{{\mathcal{B}}}
  \newcommand{\C}{{\mathcal{C}}}

  \newcommand{\F}{{\mathcal{F}}}

  \newcommand{\K}{{\mathcal{K}}}
\renewcommand{\L}{{\mathcal{L}}}

\renewcommand{\O}{{\mathcal{O}}}

  \newcommand{\T}{{\mathcal{T}}}

\newcommand{\eps}{\varepsilon}
\def\al{\alpha}
\def\be{\beta}
\def\ga{\gamma}

\def\la{\lambda}

\def\si{\sigma}

\newcommand{\bC}{\mathbb{C}}

\newcommand{\bT}{\mathbb{T}}
\newcommand{\bZ}{\mathbb{Z}}

\newcommand{\fL}{{\mathfrak{L}}}

\newcommand{\foral}{\text{ for all }}
\newcommand{\qand}{\quad\text{and}\quad}

\newcommand{\ca}{\mathrm{C}^*}

\newcommand{\cenv}{\mathrm{C}^*_{\textup{env}}}

\newcommand{\ol}{\overline}

\newcommand{\ad}{\operatorname{ad}}

\newcommand{\alg}{\operatorname{alg}}

\newcommand{\id}{{\operatorname{id}}}

\newcommand{\spn}{\operatorname{span}}

\newcommand{\sca}[1]{\left\langle#1\right\rangle} 
\newcommand{\lsca}[1]{\left[#1\right]}            
\newcommand{\rsca}[1]{\left(#1\right)}            
\newcommand{\nor}[1]{\left\Vert #1\right\Vert} 

\newcommand{\ncl}[1]{\left[ #1 \right]^{-\|\cdot\|}} 

\addtocontents{toc}{\protect\setcounter{tocdepth}{1}}

\newcommand{\ed}{\stackrel{\textup{d}}{\thicksim}}

\newcommand{\sme}{\stackrel{\textup{SME}}{\thicksim}}

\newcommand{\dee}{\, \stackrel{\textup{$\Delta$}}{\thicksim} \,}

\begin{document}

\title[Morita equivalence of C*-correspondences]{Morita equivalence of C*-correspondences passes to the related operator algebras}

\author[G.K. Eleftherakis]{George K. Eleftherakis}
\address{Department of Mathematics\\Faculty of Sciences\\University of Patras\\26504 Patras\\Greece}
\email{gelefth@math.upatras.gr}

\author[E.T.A. Kakariadis]{Evgenios T.A. Kakariadis}
\address{School of Mathematics and Statistics\\ Newcastle University\\ Newcastle upon Tyne\\ NE1 7RU\\ UK}
\email{evgenios.kakariadis@ncl.ac.uk}

\author[E.G. Katsoulis]{Elias~G.~Katsoulis}
\address{Department of Mathematics\\ East Carolina University\\ Greenville, NC 27858\\USA}
\email{katsoulise@ecu.edu}



\maketitle

\begin{abstract}
We revisit a central result of Muhly and Solel on operator algebras of C*-correspondences.
We prove that (possibly non-injective) strongly Morita equivalent C*-correspondences have strongly Morita equivalent relative Cuntz-Pimsner C*-algebras.
The same holds for strong Morita equivalence (in the sense of Blecher, Muhly and Paulsen) and strong $\Delta$-equivalence (in the sense of Eleftherakis) for the related tensor algebras.
In particular, we obtain stable isomorphism of the operator algebras when the equivalence is given by a $\sigma$-TRO.
As an application we show that strong Morita equivalence coincides with strong $\Delta$-equivalence for tensor algebras of aperiodic C*-correspondences. 
\end{abstract}

\section{Introduction}

Introduced by Rieffel in the 1970's \cite{Rie74, Rie74-2}, Morita theory provides an important equivalence relation between C*-algebras.
In the past 25 years there have been fruitful extensions to cover more general (possibly nonselfadjoint) spaces of operators.
These directions cover (dual) operator algebras and (dual) operator spaces, e.g. \cite{Ble01, BleKas08, BMP00, BMN99, Ele08, Ele10, Ele12, Ele14, EleKak16, ElePau08, EPT10, Kas09}.
There are two main streams in this endeavour.
Blecher, Muhly and Paulsen \cite{BMP00} introduced a strong Morita equivalence $\sme$, along with a Morita Theorem I, where the operator algebras $\A$ and $\B$ are symmetrically decomposed by two bimodules $M$ and $N$, i.e.
\[
\A \simeq M \otimes_\B N \qand \B \simeq N \otimes_\A M.
\]
Morita Theorems II and III for $\sme$ were provided by Blecher \cite{Ble01}. 
On the other hand Eleftherakis \cite{Ele08} introduced a strong $\Delta$-equivalence $\dee$ that is given by a generalized similarity under a TRO $M$, i.e.
\[
\A \simeq M \otimes \B \otimes M^* \qand \B \simeq M^* \otimes \A \otimes M.
\]
Although they coincide in the case of C*-algebras, relation $\dee$ is strictly stronger than relation $\sme$.
Indeed $\sme$ does not satisfy a Morita Theorem IV, even when $X$ and $Y$ are unital \cite[Example 8.2]{BMP00}.
However a Morita Theorem IV holds for $\dee$ on $\si$-unital operator algebras \cite[Theorem 3.2]{Ele14}.

The Morita context can be modified to cover other classes as well.
In their seminal paper, Muhly and Solel \cite{MuhSol00} introduced a strong Morita equivalence for C*-correspondences and formulate the following programme for the tensor algebras of C*-correspondences:
\begin{quote}
\textbf{Rigidity of SME.} Let $\T_E^+$ and $\T_F^+$ be the tensor algebras of the C*~-~correspondences $E$ and $F$.
When is it true that $E \sme F$ is equivalent to $\T_E^+ \sme \T_F^+$?
\end{quote}
The origins of this programme can be traced in the work of Arveson \cite{Arv67} on classifying dynamics by nonselfadjoint operator algebras.
In this category $\sme$ is seen as a generalized similarity, rather than a decomposition\footnote{\ 
Decompositions in this category generalize the shift equivalences for matrices.
This stream of research follows a completely different path exploited by Muhly, Pask and Tomforde \cite{MPT08}, and Kakariadis and Katsoulis \cite{KakKat14}.}.
Muhly and Solel \cite{MuhSol00} provide an affirmative answer for injective and aperiodic C*-correspondences.
The main tool for the forward direction is to establish $\sme$ for the Toeplitz-Pimsner C*-algebras $\T_E$ and $\T_F$.
An elegant construction of matrix representations for $\T_E$ and $\T_F$ is used in \cite{MuhSol00} to accomplish this.
Aperiodicity is used for the converse to ensure that an induced Morita context is implemented fiber-wise.

Our first motivation for the current paper was to remove the injectivity assumption for the forward direction of SME-rigidity.
In fact we accomplish more by directly showing an equivalence implemented by the same TRO  between representations.
As a consequence $\sme$ on C*-correspondences implies equivalence of all related operator algebras, i.e. their tensor algebras and their $J$-relative Cuntz-Pimsner algebras, for $J$ inside Katsura's ideal (Theorem \ref{T: sme}).
We highlight that the converse of the rigidity question has been exhibited to be considerably difficult even when $\sme$ is substituted by honest isomorphisms.
In this reduced form, it is better known as the Conjugacy Problem and it has been answered in several major classes of C*-correspondences, e.g. \cite{DavKak12, DavKat08, DavKat11, DavRoy11, Dor15, KatKri04, KakSha15, Sol04} to mention but a few.

We derive our motivation also from further researching $\sme$ and $\dee$.
The differences between these relations are subtle and it is natural to ask when they actually coincide.
We show that this is true for tensor algebras of aperiodic non-degenerate C*-correspondences.
In particular we prove that $E \sme F$ if and only if $\T_E^+ \dee \T_F^+$ if and only if $\T_E^+ \sme \T_F^+$ (Corollary \ref{C: aperiodic}).
This is quite pleasing as we incorporate a big class of operator algebras with approximate identities.
Recall that $\sme$ and $\dee$ are shown to be different even for unital operator algebras.

Our results read the same if $\sme$ is substituted by stable isomorphisms (Corollary \ref{C: stable}).
This follows directly from our analysis and the observation that stable isomorphism coincides with $\sme$ by a $\si$-TRO.
Notice that we do not make a distinction between unitary equivalence and isomorphism in the category of operator bimodules.
In particular we show that they coincide for C*-correspondences by viewing the linking algebra as the C*-envelope of an appropriate subalgebra (Proposition \ref{P: uneq=cc}).
Our methods then use a fundamental result of Blecher \cite{Ble97} which states that the stabilized tensor product coincides with the balanced Haagerup tensor product for non-degenerate C*-correspondences.

Although we include it in all statements, it is worth mentioning here that C*-correspondences (and thus all representations) are considered to be non-degenerate.
This is not an artifact for convenience.
Strong Morita equivalence automatically induces non-degeneracy of the C*-correspondences. 

\section{Preliminaries}

The reader should be well acquainted with the theory of operator algebras \cite{BleLeM04, Pau02}.
For an exposition on the C*-envelope s/he may refer to \cite{KakNotes}.
Furthermore the reader should be familiar with the general theory of Hilbert modules and C*-correspondences.
For example see \cite{Lan95, ManTro01} for Hilbert modules and \cite{Kat04, MuhSol98} for C*-correspondences.
We will give a brief introduction for purposes of notation and terminology.

\subsection{C*-correspondences}

A C*-correspondence ${}_A X_B$ is a right Hilbert module $X$ over $B$ along with a $*$-homomorphism $\phi_X \colon A \to \L(X)$.
It is called \emph{injective} (resp. \emph{non-degenerate}) if $\phi_X$ is injective (resp. non-degenerate).
It is called \emph{full} if $\sca{X,X} := \ol{\spn}\{\sca{x,y} \mid x,y \in X\} = B$.
It is called an \emph{imprimitivity bimodule} (or \emph{equivalence bimodule}) if it is full, injective and $\phi_X(A) = \K(X)$.

A \emph{(Toeplitz) representation} $(\pi,t)$ of ${}_A X_A$ on a Hilbert space $H$, is a pair of a $*$-homomorphism $\pi\colon A \rightarrow \B(H)$ and a linear map $t\colon X \rightarrow \B(H)$, such that
\[
\pi(a)t(x)=t(\phi_X(a)(x)) \qand t(x)^*t(y)=\pi(\sca{x,y}_X)
\]
for all $a\in A$ and $x,y\in X$. 
A representation $(\pi,t)$ automatically satisfies
\[
t(x) \pi(a) = t(x a) \foral x \in X, a \in A.
\] 
Moreover there exists a $*$-homomorphism $\psi_t \colon \K(X) \to \B(H)$ such that $\psi_t(\Theta_{x, y}) = t(x) t(y)^*$ \cite{KPW98}.
A representation $(\pi , t)$ is said to be \textit{injective} if $\pi$ is injective.
If $(\pi,t)$ is injective then $t$ is an isometry and $\psi_t$ is injective.

A crucial remark made by Katsura \cite{Kat04} is that $a \in \ker\phi_X^\perp \cap \phi_X^{-1}(\K(X))$ whenever $\pi(a) \in \psi_t(\K(X))$ for an injective representation $(\pi,t)$ of ${}_A X_A$.
This completes the analysis of Muhly-Solel \cite{MuhSol98} on covariant representations, by covering also the non-injective cases.
In more details let $J$ be an ideal in $\phi_X^{-1}(\K(X))$.
A representation $(\pi,t)$ is called \emph{$J$-covariant} if
\[
 \psi_t(\phi_X(a)) = \pi(a) \foral a\in J.
\]
The representations $(\pi,t)$ that are $J_{X}$-covariant for \emph{Katsura's ideal}
\[
J_X := \ker\phi_X^\bot \cap \phi_X^{-1}(\K(X)),
\]
are called \emph{covariant representations}. 
The ideal $J_X$ is the largest ideal on which the restriction of $\phi_X$ is injective.
An example of a covariant representation is given by taking the quotient map with respect to $\K(\F(X) J_X)$ on the Fock representation \cite{Kat04}.

The \emph{Toeplitz-Pimsner algebra} $\T_X$ is the universal C*-algebra with respect to the Toeplitz representations of ${}_A X_A$. 
The \emph{$J$-relative Cuntz-Pimsner algebra} $\O(J,X)$ is the universal C*-algebra with respect to the $J$-covariant representations of ${}_A X_A$.
The \emph{Cuntz-Pimsner algebra} $\O_X$ is the universal C*-algebra with respect to the covariant representations of ${}_A X_A$.
The \emph{tensor algebra} $\T_X^+$ in the sense of Muhly-Solel \cite{MuhSol98} is the algebra generated by the copies of $A$ and $X$ inside $\T_X$.

Due to the Fock representation, the copies of $A$ and $X$ inside $\T_X$ are isometric.
In addition $\T_X^+$ is embedded completely isometrically in $\O_X$, and $\cenv(\T_X^+) \simeq \O_X$.
This was accomplished under certain conditions by Fowler-Muhly-Raeburn \cite{FMR03}; all assumptions were finally removed by Katsoulis-Kribs \cite{KatKri06}.
Furthermore Kakariadis-Peters \cite{KakPet14} have shown that $J \subseteq J_X$ if and only if $A \hookrightarrow \O(J,X)$ if and only if $\T_X^+ \hookrightarrow \O(J, X)$.
The proof follows by \cite{KatKri06} and the diagram
\begin{align*}
\xymatrix@R=1em@C=4em{ \T_X \ar[rr] \ar[dr] & & \O_X\\
&  \O(J, X) \ar[ur] &
}
\end{align*}
where the arrows indicate canonical $*$-epimorphisms.
Therefore the $J$-relative algebras for $J \subseteq J_X$ are the only Pimsner algebras that contain an isometric copy of the C*-correspondences.
Beyond this point we ``lose'' information of the original data.

We say that a representation $(\pi,t)$ admits a gauge action $\{\be_z\}_{z \in \bT}$ if every $\be_z$ is an automorphism of $\ca(\pi,t)$ such that
\begin{equation}\label{eq:ga}
\be_z(\pi(a)) = \pi(a) \qand \be_z(t(x)) = z t(x)
\end{equation}
for all $a \in A$ and $x \in X$, and the family $\{\be_z\}_{z \in \bT}$ is point-norm continuous.
Since $\ca(\pi,t)$ is densely spanned by the monomials of the form $t(x_1)\cdots t(x_n) t(y_m)^* \cdots t(y_1)^*$, an $\eps/3$-argument implies that if $\{\be_z\}_{z \in \bT}$ is a family of $*$-homomorphisms of $\ca(\pi,t)$ that satisfies equation (\ref{eq:ga}) then it is point-norm continuous.

The Gauge-Invariant-Uniqueness-Theorem (GIUT) is a fundamental result for lifting representations of the C*-correspondence to operator algebras.
This type of result was initiated by an Huef and Raeburn for Cuntz-Krieger algebras \cite[Theorem 2.3]{HueRae97}.
Generalizations (under certain assumptions on the C*-correspondence) were given by Doplicher-Pinzari-Zuccante \cite[Theorem 3.3]{DPZ98}, Fowler-Muhly-Raeburn \cite[Theorem 4.1]{FMR03}, and Fowler-Raeburn \cite[Theorem 2.1]{FowRae99}. 
All assumptions were removed by Katsura \cite[Theorem 6.2, Theorem 6.4]{Kat04} by using a sharp analysis of cores and a conceptual argument involving short exact sequences. 
Another proof for $\O_X$ was provided by Muhly-Tomforde \cite{MuhTom04} by using a tail adding technique. 
A remarkable extension to the much broader class of pre-C*-correspondences is given by Kwa\'{s}niewski \cite{Kwa11}.
The second author \cite{Kak14} gave an alternative proof of the GIUT that treats all $J$-relative Cuntz-Pimsner algebras with $J \subseteq J_X$: \emph{a representation $(\pi, t)$ of ${}_A X_A$ lifts to a faithful representation of $\O(J, X)$ if and only if $(\pi,t)$ admits a gauge action, $\pi$ is injective and $J = \{a \in A \mid \pi(a) \in \psi_t(\K(X))\}$}.
Consequently if $(\pi,t)$ admits a gauge action and $\pi$ is injective then $\ca(\pi,t) \simeq \O(J,X)$ for
\[
J = \{a \in A \mid \pi(a) \in \psi_t(\K(X))\}. 
\]

We remark that when ${}_A X_A$ is non-degenerate then it suffices to consider just the representations $(\pi, t)$ with $\pi$ non-degenerate.
Indeed it is easy to check that $\pi(A)$ carries a c.a.i. for $\ca(\pi,t)$ and we can pass to an appropriate Hilbert subspace where $\pi$ acts non-degenerately.
This will always be the case in the current paper.

\subsection{Tensor products}

Let us recall the following results of Blecher \cite{Ble97} concerning tensor products.
For this subsection let us fix a right $A$-module $X$ and a C*-correspondence ${}_A Z_B$.
We further assume that $Z$ is non-degenerate; otherwise all that follow hold for the essential part $\ncl{\phi_Z(A)Z}$ of $Z$.

Blecher \cite[Theorem 3.1]{Ble97} has shown that the right Hilbert modules are asymptotic summands of the free modules $C_n(A) = \sum_{i=1}^n A$, i.e. for $X_A$ there are completely contractive $A$-module maps
\[
\phi_\al \colon X \to C_{n(\al)}(A) \qand \psi_\al \colon C_{n(\al)}(A) \to X
\]
such that $\psi_\al \phi_\al \to \id_X$ strongly.
One of the main consequences \cite[Theorem 4.3]{Ble97} is that the stabilized Haagerup tensor product $X \otimes_A^h Y$ is completely isometrically isomorphic to the stabilized Hilbert-module tensor product $X \otimes_A Y$.
This is derived by following the diagonals in the diagram
\[
\xymatrix@R=5mm@C=20mm{
X \otimes_A^h Y \ar@{-->}[dr]_{\phi_\al} \ar[rr]^{\id} \ar@{-->}[dd] & & X \otimes_A^h Y \\
& C_{n(\al)}(Y) \ar[ur]_{\psi_\al} \ar@{-->}[dr]^{\psi_\al} & \\
X \otimes_A Y \ar[ur]^{\phi_\al} \ar[rr]^{\id} & & X \otimes_A Y \ar[uu]
}
\]
and using the fact that $C_n(A) \otimes_A^h Y \simeq C_n(Y) \simeq C_n(A) \otimes_A Y$ completely isometrically.

\subsection{Ternary rings of operators}

A \emph{ternary ring of operators (TRO)} $M$ is a subspace of some $\B(H)$ such that $MM^*M \subseteq M$.
It then follows that $M$ is an imprimitivity bimodule over $A = \ncl{MM^*}$ and $B = \ncl{M^*M}$.
Every C*-correspondence $X$ (on some $A$) is a TRO.
However it gives an equivalence between $\K(X)$ and $\sca{X, X}$ which may differ from $A$ in principle.
Moreover, for the TRO $M$ there are two nets
\begin{align*}
a_t =
\sum_{i=1}^{l_t} m_i^t (m_i^t)^*
\qand
b_\lambda =
\sum_{i=1}^{k_\lambda } (n_i^\lambda)^* n_i^\lambda
\end{align*}
for $m_i^t, n_j^\lambda \in M$ such that all $[m_1^t, m_2^t, \dots, m_{l_t}^t]$ and $[(n_1^\la)^*, (n_2^\la)^*, \dots, (n_{k_\la}^\la)^*]$ are row contractions and
\[
\lim _t a_t m = m \qand \lim_\lambda m b_\lambda = m
\]
for all $m\in M$; see for example the proof of \cite[Theorem 6.1]{BMP00}.
In particular the C*-algebras $A = \ncl{MM^*}$ and $B = \ncl{M^*M}$ are $\si$-unital if and only if the nets $(a_t)$ and $(b_\la)$ can be chosen to be sequences; see \cite[Lemma 2.3]{Bro77}.
If any of the above happens then we will say that $M$ is a \emph{$\si$-TRO}.
These approximate identities provide an efficient tool for the study of strong Morita equivalence.
We include the following well-known technique for future reference.

\begin{lemma}\label{L: H-technique}
Let $M \subseteq \B(H,K)$ be a TRO and let $A = \ncl{MM^*}$.
Let $X \subseteq \B(K)$ be an operator right module over $A$.
Then $X \otimes_A^h M \simeq \ncl{X M}$.
\end{lemma}

\begin{proof}
Let the completely contractive $A$-balanced bilinear map $X \times M \to \ncl{X M}$ defined by $(x, m) \mapsto x m$.
This induces a completely contractive map
\[
\Phi \colon X \otimes ^h_A M \mapsto \ncl{X M}: x \otimes_A m \to x m.
\]
Let $n_\la^* = [(n_1^\la)^*, \dots, (n_{k_\la}^\la)^*]$ be a net provided by $M$; i.e.  $n_i^\la \in M$, $\nor{n_\la} \leq 1$ and
\[
\lim_\la m n_\la^* n_\la = m \foral m\in M.
\]
Fix $x_1, \dots, x_n\in X$ and $m_1, \dots, m_n \in M$.
For $\eps >0$ there exists $\la$ such that
\begin{align*}
\| \sum_{i=1}^n x_i \otimes m_i \| -\eps
& \leq
\| \sum_{i=1}^n x_i \otimes(m_i n_\la^* n_\la) \|
 =
\| \sum_{i=1}^n (x_i m_i n_\la^*) \otimes n_\la \| \\
& \leq
\| \sum_{i=1}^n x_i m_i n_\la^* \|
\leq
\| \sum_{i=1}^n x_i m_i \|.
\end{align*}
Therefore $\Phi$ is isometric.
A similar argument holds for the matrix norms and the proof is complete.
\end{proof}

\subsection{Strong Morita equivalence for operator algebras}

There is a rich Morita theory for C*-algebras produced by Rieffel \cite{Rie74}, Brown \cite{Bro77} and Brown-Green-Rieffel \cite{BGR77}.
Morita theory extends in various ways to nonselfadjoint operator algebras with the most influential being the direction of Blecher-Muhly-Paulsen \cite{BMP00} for algebras that admit an approximate unit.
For the sake of completeness we include a short description.

Suppose that $\A$ and $\B$ are operator algebras with c.a.i.'s and let the bimodules ${}_\A M_\B$ and ${}_\B N_\A$.
Following \cite[Definition 3.1]{BMP00}, a \emph{Morita context} is a system $(\A,\B, M, N, (\cdot,\cdot), [\cdot,\cdot])$ where the maps
\[
(\cdot,\cdot) \colon M \times N \to \A \qand [\cdot,\cdot] \colon N \times M \to \B
\]
are completely bounded, bilinear and balanced satisfying the following three properties:
\begin{enumerate}
\item[(A)] For all $m_1, m_2 \in M$ and $n_1, n_2 \in N$ we have that $(m_1, n_1) m_2 = m_1 [n_1, m_2]$ and $[n_1, m_1] n_2 = n_1 (m_1, n_2)$;
\item[(G)] The map from $M \otimes^h N$ to $\A$ induced by $(\cdot, \cdot)$ is a complete quotient map onto $\A$;
\item[(P)] The map from $N \otimes^h M$ to $\B$ induced by $[\cdot, \cdot]$ is a complete quotient map onto $\B$.
\end{enumerate}
If such a system exists then $\A$ and $\B$ are called \emph{strongly Morita equivalent} (notation, $\A \sme \B$) \cite[Definition 3.6]{BMP00}.
The initials (A), (G) and (P) stand for associativity, generator and projective respectively.
When $\A$ and $\B$ are C*-algebras then this definition coincides with that of \cite{Rie74}, where $M$ \emph{can be chosen} to be an imprimitivity bimodule and $N$ be $M^*$.

There is a concrete version of the Morita context.
One of the main ingredients concerns the form of the decomposition of two algebras.
(Notice here that we do not assume a priori that the operator algebras have an approximate unit.)

\begin{definition}
Let $\A$ and $\B$ be operator algebras.
We say that they are \emph{decomposable} (notation, $\A \ed \B$) if there are non-degenerate completely isometric representations
\[
\al \colon \A \to \B(H) \qand \be \colon \B \to \B(K)
\]
and bimodules ${}_{\al(\A)} M_{\be(\B)} \subseteq \B(K, H)$ and ${}_{\be(\B)} N_{\al(\A)} \subseteq \B(H, K)$ such that
\[
\al(\A) = \ncl{M \cdot N} \qand \be(\B) = \ncl{N \cdot M}.
\]
\end{definition}

By definition $\ed$ implies a Morita context $(\A, \B, M, N, \rsca{\cdot, \cdot}, \lsca{\cdot, \cdot})$ that satisfies just property (A) in the sense of Blecher-Muhly-Paulsen \cite{BMP00}.
It is unclear (and probably not true) that $\ed$ is transitive in general.
Remarkably though $\ed$ is an equivalence relation on unital operator algebras \cite[Proposition 3.3]{BMP00}.
In fact Blecher-Muhly-Paulsen \cite{BMP00} show that when $\A$ and $\B$ have c.a.i.'s then being strongly Morita equivalent coincides with having
\[
\A \simeq M \otimes_{\B}^h N \qand \B \simeq N \otimes_{\A}^h M.
\]
This depends on an elegant decomposition of the approximate unit.
Apparently the existence of approximate units is necessary.
It can be shown that $\sme$ is equivalent to having $\ed$ and row contractions from $M$ and $N$ that reconstruct c.a.i.'s for $\A$ and $\B$.

The importance of Morita Theory appears through the Morita Theorems.
Let us give a description that depicts the essential points.
Morita Theorem I suggests that Morita equivalence implies an equivalence of the representations by tensoring (within the category) with the bimodules of the Morita context.
Morita Theorem II suggests that an equivalence of representations gives rise to a Morita context, whereas Morita Theorem III implies that such equivalences are given just by tensoring with appropriate bimodules.
Finally Morita Theorem IV suggests that if the algebras satisfy a countability condition (the c.a.i.'s are countable) then Morita equivalence coincides with stable isomorphism.
Recall that $\A$ and $\B$ are \emph{stably isomorphic} if $\A \otimes \K \simeq \B \otimes \K$ by a completely isometric isomorphism, where $\K$ is the compacts on $\ell^2$ and the tensor product is the spatial one.

Strong Morita equivalence for approximately unital operator algebras satisfies the first three parts of Morita Theory \cite{Ble01, BMP00}. 
In particular, two approximately unital operator algebras are strongly Morita equivalent if and only if their categories of left operator modules are equivalent (via a completely contractive functor) \cite{BMP00}.

Eleftherakis \cite{Ele14} introduced a stronger notion of equivalence.
Two operator algebras $\A$ and $\B$ are \emph{strongly $\Delta$-equivalent} (notation, $\A \dee \B$) if there are non-degenerate completely isometric representations
\[
\al \colon \A \to \B(H) \qand \be \colon \B \to \B(K)
\]
and a TRO $M \subseteq \B(K, H)$ such that
\[
\al(\A) = \ncl{M \be(\B) M^*} \qand \be(\B) = \ncl{M^* \al(\A) M}.
\]
In \cite{Ele14} it is shown that $\dee$ is an equivalence relation.
Even though $\dee$ is originally defined on approximately unital operator algebras in \cite{Ele14}, the elements we record here still hold for non-unital cases.
This is exhibited in \cite{EleKak16} where $\dee$ is considered for operator spaces.

As noted $\dee$ and $\sme$ coincide with the usual strong Morita equivalence when restricted to C*-algebras.
However a fundamental difference between $\dee$ and $\sme$ concerns stable isomorphisms when passing to general operator spaces.
Due to \cite{BGR77}, Morita Theorem IV holds when $\A$ and $\B$ are $\si$-unital C*-algebras.
This is not true for $\sme$ on nonselfadjoint operator algebras, even when the operator algebras are unital \cite[Example 8.2]{BMP00}.
On the other hand it is shown in \cite[Theorem 3.2]{Ele14} that $\dee$ on operator algebras coincides with stable isomorphism under the appropriate $\si$-unital condition.
As a consequence $\dee$ is strictly stronger than $\sme$.

\section{Bimodule structure}

A key role in Morita Theory is played by the linking algebra.
This construction induces an operator bimodule structure on C*-correspondences.
In this section we use it to show that isomorphism of C*-correspondences is preserved when passing to the category of operator bimodules, and vice versa.

Every right Hilbert module $X_A$ comes with an operator module structure.
This can be verified by seeing $X_A$ inside its \emph{augmented linking algebra}
\[
\begin{bmatrix} A & X^* \\ X & \L(X) \end{bmatrix} :=
\{ \begin{bmatrix} a & y^* \\ x & u \end{bmatrix} \mid a \in A, x, y \in X, u \in \L(X) \}
\]
where $X^*$ is the adjoint module $\K(X,A)$ of $X$ and the multiplication rule is given by
\[
\begin{bmatrix} a_1 & y_1^* \\ x_1 & u_1 \end{bmatrix}
\cdot
\begin{bmatrix} a_2 & y_2^* \\ x_2 & u_2 \end{bmatrix}
=
\begin{bmatrix} a_1a_2 + \sca{y_1, x_2}  & (y_2 a_1^*)^* + (u_2^*(y_1))^* \\ x_1 a_2 + u_1(x_2) & \Theta_{x_1, y_2} + u_1 u_2 \end{bmatrix} .
\]
The augmented linking algebra becomes a C*-algebra over an operator norm when the matrices are seen to act on the right Hilbert module $A + X$.
The \emph{linking algebra} is defined as the C*-subalgebra
\[
\fL(X) = \begin{bmatrix} A & X^* \\ X & \K(X) \end{bmatrix}.
\]
When $X$ is a C*-correspondence over $A$ then it admits the operator bimodule structure by viewing the left action as the matrix multiplication 
\[
\begin{bmatrix} 0 & 0 \\ \phi_X(a) x & 0 \end{bmatrix}
=
\begin{bmatrix} 0 & 0 \\ 0 & \phi_X(a) \end{bmatrix} \cdot \begin{bmatrix} 0 & 0 \\ x & 0 \end{bmatrix}.
\]

\begin{proposition}\label{P: uneq=cc}
Let ${}_A X_A$ and ${}_B Y_B$ be C*-correspondences.
Then ${}_A X_A$ and ${}_B Y_B$ are completely isometric as operator bimodules if and only if ${}_A X_A$ and ${}_B Y_B$ are unitarily equivalent as C*-correspondences.
\end{proposition}

The first observation is that if $v \in \L(X, Y)$ is invertible then the polar decomposition $v = u |v|$ gives a unitary $u \in \L(X,Y)$.
If $v$ is further a left module map then so is $u$.
Therefore we now restrict our attention just to right Hilbert modules.
We will use the following lemma.

\begin{lemma}\label{L: cenv linking}
Let $X_A$ be a right Hilbert module.
Then we have that
\[
\cenv(\begin{bmatrix} A & 0 \\ X & 0 \end{bmatrix}) = \fL(X) .
\]
\end{lemma}

\begin{proof}
For simplicity let us write $\A(X)$ for the operator algebra $\begin{bmatrix} A & 0 \\ X & 0 \end{bmatrix}$.
It is clear that $\fL(X)$ is a C*-cover of $\A(X)$.
Therefore there exists a unique *-epimorphism $\Phi \colon \fL(X) \to \cenv(\A(X))$ such that
\[
\Phi(\begin{bmatrix} a & 0 \\ x & 0 \end{bmatrix}) = i(a) + i(x) \foral a \in A, x \in X
\]
for the embedding $i \colon \A(X) \to \cenv(\A(X))$.
Let $H$ be a Hilbert space where $\cenv(\A(X))$ acts non-degenerately.

Since $A$ acts non-degenerately on the right of $X$ and $\K(X)$ acts non-degene\-ra\-tely on the left of $X$ we can choose a c.a.i. for $\fL(X)$ of the form $(a_i \oplus k_i)$ for a c.a.i. $(a_i)$ of $A$ and a c.a.i. $(k_i)$ for $\K(X)$. 
Since $\Phi$ is surjective then $(\Phi(a_i) + \Phi(k_i))$ is a c.a.i. for $\cenv(\A(X))$.
Write
\[
P := \text{sot-}\lim \Phi(a_i) \qand Q = \text{sot-}\lim_i \Phi(k_i).
\]
Therefore we obtain $\Phi(x) = Q \Phi(x) P$ for all $x \in X$.
We write
\[
\pi = P \Phi|_A P \, , \,\, t = Q \Phi|_X P \, , \,\, \psi = Q \Phi|_{\K(X)} Q.
\]
Now $P$ and $Q$ are complementary projections and thus we get that
\[
\Phi(\begin{bmatrix} a & y^* \\ x & k \end{bmatrix}) = \begin{bmatrix} \pi(a) & t(y)^* \\ t(x) & \psi(k) \end{bmatrix}.
\]
It is clear that $\pi$ and $\psi$ are *-homomorphisms such that $\psi(k) t(x) \pi(a) = t(k(x)a)$ and that both $\pi$ and $t$ are complete isometries as $A, X \subseteq \A(X)$.
We want to show that $\ker\Phi = (0)$. 
To reach contradiction let
\[
0 \neq f = \begin{bmatrix} a & y^* \\ x & k \end{bmatrix} \in \ker\Phi.
\]
By applying $\Phi$ on $f^*f$ and restricting to the $(1,1)$-entry we get that $\pi(a^*a + \sca{x,x}) = 0$.
But $\pi$ is a complete isometry as $A \subseteq \A(X)$ and thus $a=0$ and $x = 0$.
Similarly applying to $f f^*$ gives that $y = 0$.
By applying on $f g$ for any $g \in \fL(X)$ we also get that $t(k(z)) = 0$ for all $z \in X$.
Hence we have that $k(z) = 0$ for all $z \in X$, i.e. $k=0$.
This is a contradiction, and the proof is complete.
\end{proof}

\noindent \emph{Proof of Proposition \ref{P: uneq=cc}.}
If there is a unitary equivalence $(\ga, u)$ then it is clear that $\fL(X)$ and $\fL(Y)$ are $*$-isomorphic.
Hence $X$ and $Y$ are completely isometric as operator bimodules.

Conversely suppose there is completely isometric right module map $(\ga, u)$; we need to show that $u$ is adjointable.
As in \cite[Remark 3.6.1]{BleLeM04} we get the completely isometric isomorphism
\[
j \colon \A(X) \to \A(Y)
:
\begin{bmatrix} a & 0 \\ x & 0 \end{bmatrix} \mapsto \begin{bmatrix} \ga(a) & 0 \\ u(x) & 0 \end{bmatrix}.
\] 
By Lemma \ref{L: cenv linking} then the linking algebras are $*$-isomorphic by some $\Phi$.
Let $v \colon Y \to X$ such that
\[
\begin{bmatrix} 0 & 0 \\ v(y) & 0 \end{bmatrix}
=
\Phi^{-1}(\begin{bmatrix} 0 & 0 \\ y & 0 \end{bmatrix}).
\]
Then we obtain
\begin{align*}
\ga(\sca{x, v(y)}_X)
& =
\Phi(\sca{x, v(y)}_X)
=
\Phi(x)^* y
=
\sca{u(x), y}_Y
\end{align*}
where we omit the zero entries of the matrices.
Then $u^* = v$ and the proof is complete.
\hfill{$\qedsymbol$}

\begin{remark}
We remark that Proposition \ref{P: uneq=cc} does not hold for bounded bimodule maps.
Dor-On \cite{Dor15} illustrates this by examining a particular class of C*-correspondences related to weighted partial systems.
In particular it is shown that unitary equivalence and bounded isomorphisms correspond to different notions of equivalences of the original data.
Even more they reflect isometric and bounded, respectively, isomorphisms of the tensor algebras.
\end{remark}

\section{Strong Morita equivalence for C*-correspondences}

Muhly and Solel \cite{MuhSol00} initiated the study of strong Morita equivalence for C*-correspondences.
Namely ${}_A E_A$ and ${}_B F_B$ are \emph{strongly Morita equivalent} (notation, $E \sme F$) if there exists a TRO ${}_A M_B$ such that
\[
E \otimes_A M \simeq M \otimes_B F \qand M^* \otimes_A E \simeq F \otimes_B M^*.
\]
Since the tensor norm is sub-multiplicative we have that strongly Morita equivalent C*-correspondences must be non-degenerate.
Furthermore strong Morita equivalence coincides with having a TRO ${}_A M_B$ such that
\[
E \simeq M \otimes_B F \otimes_B M^* \qand F \simeq  M^* \otimes_A E \otimes_A M.
\]
Again $E$ and $F$ must be non-degenerate.
On the other hand if $E$ and $F$ are non-degenerate then it is easy to check that the $E$-equations give the $F$-equations.

In analogy to $\sme$ on C*-algebras, we obtain an equivalence between representations of strongly Morita equivalent C*-correspondences.

\begin{proposition}\label{P: same repn}
Let ${}_A E_A$ and ${}_B F_B$ be strongly Morita equivalent by a TRO $M$.
Then for every non-degenerate (injective) representation $(\pi, t)$ of ${}_A E_A$ on a Hilbert space $K$ there exists a non-degenerate (resp. injective) representation $(\si, s)$ of ${}_B F_B$ on a Hilbert space $H$ and a TRO-representation $\phi$ of $M$ in $\B(H, K)$ such that
\[
\ncl{t(E)} = \ncl{\phi(M) s(F) \phi(M)^*}
\]
and
\[
\ncl{s(F)} = \ncl{\phi(M)^* t(E) \phi(M)}.
\]
If $(\pi,t)$ admits a gauge action then so does $(\si,s)$.
\end{proposition}

\begin{proof}
To avoid technical notation we show the dual statement.
That is, given a non-degenerate (injective) representation $(\si, s)$ of ${}_B F_B$ acting on $H$, we will construct the required $(\pi, t)$ and $\phi$.
Let $K = M \otimes_B H$ and define the representation
\[
\phi \colon M \to \B(H, K) \text{ such that } \phi(x)\xi = x \otimes \xi .
\]
Then $\phi$ is a TRO representation of $M$.
Furthermore we have an induced non-degenerate (resp. injective) representation
\[
\pi \colon A \to \B(K) \text{ such that } \pi(a) x \otimes \xi = (ax) \otimes \xi.
\]
Since $E \simeq M \otimes_B F \otimes_B M^*$ we can define
\[
t \colon E \to \B(K) \text{ such that } t(m \otimes f \otimes n^*) x \otimes \xi = m \otimes (s(f) \si(n^*x) \xi).
\]
Existence of $t$ follows once we show that
\[
t(m_1 \otimes f_1 \otimes n_1^*)^*t(m_2 \otimes f_2 \otimes n_2^*) = \pi(\sca{m_1 \otimes f_1 \otimes n_1^*, m_2 \otimes f_2 \otimes n_2^*}.
\]
In this case $t$ will be norm-decreasing on finite sums of $m \otimes f \otimes n^* \in E$ and thus can be extended to the entire of $E$.
A straightforward computation reveals that
\begin{align*}
\sca{t(m_1 \otimes f_1 \otimes n_1^*)x_1 \otimes \xi_1, t(m_2 \otimes f_2 \otimes n_2^*) x_2 \otimes \xi_2}_K & = \\
& \hspace{-5cm} =
\sca{m_1 \otimes (s(f_1) \si(n_1^*x_1) \xi_1), m_2 \otimes (s(f_2) \si(n_2^*x_2) \xi_2)}_K \\
& \hspace{-5cm} =
\sca{s(f_1) \si(n_1^*x_1) \xi_1, \si(m_1^* m_2) s(f_2) \si(n_2^*x_2) \xi_2}_H \\
& \hspace{-5cm} =
\sca{\xi_1, \si(x_1^* n_1 \sca{f_1, m_1^* m_2 f_2} n_2^* x_2) \xi_2}_H
\end{align*}
and on the other hand we have that
\begin{align*}
\sca{x_1 \otimes \xi_1, \pi(\sca{m_1 \otimes f_1 \otimes n_1^*, m_2 \otimes f_2 \otimes n_2^*}) (x_2 \otimes \xi_2)}_K & = \\
& \hspace{-5cm} =
\sca{x_1 \otimes \xi_1, \pi(n_1 \sca{f_1, m_1^* m_2 f_2} n_2^*) (x_2 \otimes \xi_2)}_K \\
& \hspace{-5cm} =
\sca{x_1 \otimes \xi_1, (n_1 \sca{f_1, m_1^* m_2 f_2} n_2^* x_2) \otimes \xi_2}_K \\
& \hspace{-5cm} =
\sca{\xi_1, \si(x_1^* n_1 \sca{f_1, m_1^* m_2 f_2} n_2^* x_2) \xi_2}_H .
\end{align*}

By construction we have that $\phi(m) s(f) \phi(n)^* = t(m \otimes f \otimes n^*)$ and that $\phi(n)^* t(e) \phi(m) = s(n^* \otimes e \otimes m)$.
Consequently we obtain
\[
\ncl{\phi(M) s(F) \phi(M)^*} \subseteq \ncl{t(E)}
\]
and
\[
\ncl{\phi(M)^* t(E) \phi(M)} \subseteq \ncl{s(F)}.
\]
Since $A$ and $B$ act non-degenerately we obtain equalities.

To avoid technical notation we will show the second item just when $\pi$ is injective (and thus $t$ is isometric).
For the general case substitute $E$ by $\ncl{t(E)}$ and $F$ by $\ncl{s(F)}$ in what follows.
We will also make the simplifications
\[
t(E) = E \subseteq \fL(E) \subseteq \ca(\pi,t).
\]
Suppose that $(\pi,t)$ admits a gauge action $\{\be_z\}$.
Then every $\be_z$ induces a representation $(\be_z|_A, \be_z|_E)$ on $E$.
Since $F \simeq M^* \otimes_A E \otimes_A M$ we can define the mapping
\[
\ga_z = \id_{M^*} \otimes \be_z \otimes \id_{M} \colon F \to F 
\]
where for simplicity we don't write the unitary of the equivalence.
Then $(\id_B, \ga_z)$ induces an injective representation of $F$.
We have to show that it induces a gauge action on the entire $\ca(\si, s)$.
For the TRO $M$ fix the net $n_\la^* = [(n_1^\la)^*, \dots, (n_{k_\la}^\la)^*]$ such that
\[
\lim_\la m n_\la^* n_\la = m \foral m \in M.
\]
Recall that every $n_\la$ is a contraction.
Let the completely contractive maps
\[
\phi_\la \colon \ca(\si, s) \to M_{k_\la}(\ca(\pi,t)) \text{ such that } \phi_\la = \ad_{\phi(n_\la)^*}
\]
and
\[
\psi_\la \colon M_{k_\la}(\ca(\pi,t)) \to \ca(\si, s) \text{ such that } \psi_\la = \ad_{\phi(n_\la)}.
\]
Then for every $f \in \ca(\si,s)$ of the form
\[
f = \phi(m_1)^* t(\xi_1) \dots t(\xi_{k}) t(\eta_{l})^* \dots t(\eta_1)^* \phi(m_2)
\]
with $\xi_i, \eta_j \in E$ and $m_1, m_2 \in M$ we get that $\lim_\la \psi_\la \phi_\la(f) = f$.
By iterating we can extend $\ga_z$ to be defined on all elements $f$ of this form and so that
\[
\phi_\la \ga_z(f) = (\be_z \otimes \id_{k_\la}) \phi_\la(f).
\]
Therefore we obtain
\begin{align*}
\nor{\ga_z(f)}
& =
\lim_\la \nor{\psi_\la \phi_\la \ga_z(f)} 
 \leq
\limsup_\la \nor{\phi_\la \ga_z(f)} \\
& =
\limsup_\la \nor{(\be_z \otimes \id_{k_\la}) \phi_\la(f)} 
 \leq 
\limsup_\la \nor{\phi_\la(f)} 
\leq \nor{f}.
\end{align*}
Applying for $\ga_{\ol{z}}$ gives $\nor{\ga_z(f)} = \nor{f}$.
Linearity allows to use the same arguments when we consider finite sums of elements of the form of $f$.
However such finite sums span a dense subspace of $\ca(\si,s)$ and thus $\ga_z$ extends to a $*$-isomorphism of $\ca(\si,s)$.
An $\eps/3$ argument shows that $\{\ga_z\}$ is in particular point-norm continuous and the proof is complete.
\end{proof}

This construction respects $J$-covariance of the representations.

\begin{proposition}\label{P: same covariance}
Let ${}_A E_A$ and ${}_B F_B$ be strongly Morita equivalent by a TRO $M$.
Let $(\pi,t)$ be a non-degenerate representation of ${}_A E_A$ and let $(\si,s)$ be the non-degenerate representation constructed in Proposition \ref{P: same repn}.
If $\ca(\pi,t) \simeq \O(J, E)$ for $J \subseteq J_E$ then $M^* J M \subseteq J_F$ and $\ca(\si,s) \simeq \O(M^* J M, F)$.
\end{proposition}

\begin{proof}
First we show that $M^* J_E M = J_F$.
One consequence of the Gauge-Invariant-Uniqueness-Theorem, as presented in \cite{Kak14}, is that
\[
J_E = \{a \in A \mid \pi(a) \in \psi_t(\K(E))\}
\]
for any non-degenerate injective covariant $(\pi,t)$ that admits a gauge action.
Fix such a $(\pi,t)$, so that $\O_E = \ca(\pi,t)$.
Let $(\si,s)$ be as in Proposition \ref{P: same repn}.
By construction we get that
\begin{align*}
\ncl{\phi(M)^* t(E) t(E)^* \phi(M)} 
& =
\ncl{\phi(M)^* t(E) \pi(A) t(E)^* \phi(M)} \\
& =
\ncl{\phi(M)^* t(E) \phi(M) \phi(M)^* t(E)^* \phi(M)} \\
& = 
\ncl{s(F) s(F)^*}.
\end{align*}
As $t(E)t(E)^*$ is dense in $\psi_t(\K(E))$, and likewise for $F$, we get that
\[
\psi_s(\K(F)) = \ncl{\phi(M)^* \psi_t(\K(E)) \phi(M)}.
\]
In a similar way we obtain the dual
\[
\psi_t(\K(E)) = \ncl{\phi(M) \psi_s(\K(F)) \phi(M)^*}.
\]
For $b = m^* a n \in M^* J_E M$ we then get
\[
M b M^* \in \ncl{MM^* J_E MM^*} = J_E
\]
as $\ncl{MM^*} = A$.
Therefore we have
\[
\phi(M) \si(b) \phi(M)^* = \pi(M b M^*)
\subseteq 
\psi_t(\K(E)) = \ncl{\phi(M) \psi_s(\K(F)) \phi(M)^*}.
\]
Consequently we derive
\[
\si(B b B) = \ncl{\phi(M)^* \phi(M) \si(b) \phi(M)^* \phi(M)} \subseteq \psi_s(\K(F)),
\]
and hence $\si(b) \in \psi_s(\K(F))$.
Since $\si$ is injective we then automatically get that $b \in J_F$; thus $M J_E M^* \subseteq J_F$.
In a dual way we obtain $M^* J_F M \subseteq J_E$.
Combining those gives the required equality.
Now for $b \in J_F$ we have $M b M^* \in J_E$ and thus $\pi(M b M^*) \in \psi_t(\K(E))$.
Following the same arguments as above we obtain that $\si(b) \in \psi_s(\K(F))$ and therefore $(\si, s)$ is covariant.
 
Once we deal with $J_E$ we can run the same arguments to complete the proof.
This follows by a consequence of the Gauge-Invariant-Uniqueness-Theorem, as presented in \cite{Kak14}, i.e. if $J \subseteq J_E$ then
\[
J = \{a \in A \mid \pi(a) \in \psi_t(\K(E))\}
\]
for any non-degenerate injective $J$-covariant $(\pi,t)$ that admits a gauge action.
The key is to notice that if $J \subseteq J_E$ then $M^* J M \subseteq M^* J_E M = J_F$ from the first part.
Now the arguments follow mutatis mutandis.
\end{proof}

The following corollary gives a necessary and sufficient condition for strong Morita equivalence in the spirit of strong $\Delta$-equivalence \cite{Ele14}.
Recall that if ${}_A X_B$ is an operator bimodule then a \emph{faithful CES-representation} is a triple $(\pi, t, \si)$ of completely contractive maps $\pi \colon A \to \B(H)$, $t \colon X \to \B(K,H)$ and $\si \colon B \to \B(K)$ where $t$ is completely isometric \cite[Section 3.3]{BleLeM04}.

\begin{corollary}\label{C: ns cond sme}
Let ${}_A E_A$ and ${}_B F_B$ be non-degenerate C*-correspondences.
Then ${}_A E_A$ and ${}_B F_B$ are strongly Morita equivalent if and only if there is
\begin{enumerate}
\item a non-degenerate faithful CES-representation $(\pi, t, \pi)$ of ${}_A E_A$ on a Hilbert space $K$;
\item a non-degenerate faithful CES-representation $(\si, s, \si)$ of ${}_B F_B$ on a Hilbert space $H$; and 
\item a TRO $M \subseteq \B(H, K)$ such that
\[
t(E) = \ncl{M s(F) M^*} \qand s(F) = \ncl{M^* t(E) M},
\]
and
\[
\pi(A) = \ncl{M \si(B) M^*} \qand \si(B) = \ncl{M^* \pi(A) M},
\]
(where possibly $\ncl{MM^*} \neq \pi(A)$ or $\ncl{M^*M} \neq \si(B)$).
\end{enumerate}
\end{corollary}

\begin{proof}
For the forward implication suppose that $E \sme F$. 
Let $\fL(E)$ act non-degenerate\-ly and faithfully on some $K$.
Then there is an induced non-degenerate representation $(\pi, t)$ of the C*-correspondence $E$ which trivially is a faithful CES-representation.
An application of Proposition \ref{P: same repn} finishes this direction.

Conversely, notice that by substituting $M$ with $\ncl{\pi(A)M\si(B)}$ we get an $A$-$B$-imprimitivity bimodule for which the relations continue to hold.
Therefore without loss of generality we may assume that $M$ is an $A$-$B$-imprimitivi\-ty bimodule.
By Lemma \ref{L: H-technique} we then obtain
\[
E \simeq M \otimes_B^h F \otimes_B^h M^* \qand F \simeq M^* \otimes_A^h E \otimes_A^h M
\]
as operator spaces.
However the Haagerup tensor product coincides with the interior tensor product by \cite[Theorem 4.3]{Ble97}.
Therefore we get that
\[
E \simeq M \otimes_B F \otimes_B M^* \qand F \simeq M^* \otimes_A E \otimes_A M
\]
as operator spaces.
Notice that the isomorphisms are completely isometric isomorphisms in the operator modules category.
Therefore Proposition \ref{P: uneq=cc} applies to give that the isomorphisms induce unitary equivalences, and the proof is complete.
\end{proof}

\begin{theorem}\label{T: sme}
Let ${}_A E_A$ and ${}_B F_B$ be strongly Morita equivalent C*-corre\-spondences by a TRO $M$. 
Then:
\begin{enumerate}
\item $\T_E \sme \T_F$.
\item $\O(J, E) \sme \O(M^*JM, F)$ for every $J \subseteq J_E$.
\item $\O_E \sme \O_F$.
\item $\T_E^+ \dee \T_F^+$ in the sense of Eleftherakis \cite{Ele14}.
\item $\T_E^+ \sme \T_F^+$ in the sense of Blecher-Muhly-Paulsen \cite{BMP00}.
\end{enumerate}
\end{theorem}

\begin{proof}
For items (i)-(iii) it suffices to show item (ii).
Fix a non-degenerate injective representation $(\pi,t)$ of ${}_A E_A$ such that $\ca(\pi, t) \simeq \O(J, E)$.
By the Gauge-Invariant-Uniqueness-Theorem then $(\pi,t)$ is an injective $J$-covariant representation that admits a gauge action.
By Proposition \ref{P: same repn} we get an injective representation $(\si, s)$ for ${}_B F_B$ that admits a gauge action.
Proposition \ref{P: same covariance} implies that $(\si,s)$ is $M^* JM$-covariant and therefore $\O(M^* J M, F) \simeq \ca(\si,s)$ by the Gauge-Invariant-Uniqueness-Theorem.
A direct computation gives that
\begin{align*}
\ncl{\phi(M)^* t(E^{\otimes 2}) \phi(M)} 
& = 
\ncl{\phi(M)^*t(E) \pi(A) t(E) \phi(M)} \\
& = 
\ncl{\phi(M)^*t(E) \phi(M) \phi(M)^* t(E) \phi(M)} \\
& = 
s(F^{\otimes 2}).
\end{align*}
Similarly we have this for all tensor powers and their adjoints.
Since $\ca(\si,s)$ is generated by $s(F^{\otimes k}) s(F^{\otimes l})^*$ for $k, l \in \bZ_+$ we get that
\[
\ca(\si, s) = \ncl{\phi(M)^* \ca(\pi,t) \phi(M)}.
\]
Likewise we get the dual $\ca(\pi, t) = \ncl{\phi(M)^* \ca(\si,s) \phi(M)}$.
Therefore we deduce that $\O(J, E) \sme \O(M^* J M, F)$.

Recall that $\T_E^+ \simeq \ol{\alg}\{ \pi(A), t(E) \}$ and that $\T_F^+ \simeq \ol{\alg}\{ \si(B), s(F) \}$ since $J \subseteq J_E$ and $M^* J M \subseteq J_F$.
Notice that the algebraic relations above imply also that
\[
\ol{\alg}\{ \si(B), s(F) \} = \ncl{\phi(M)^* \ol{\alg}\{ \pi(A), t(E) \} \phi(M)}
\]
as well as the dual relation.
Therefore we obtain $\T_E^+ \dee \T_F^+$ as operator algebras.
By \cite{Ele14} we then get that $\T_E^+ \sme \T_F^+$ since these algebras attain approximate units.
\end{proof}

\begin{remark}
An alternative proof of item (iii) of Theorem \ref{T: sme} can be given by using the C*-envelope.
Item (i) implies item (iv).
In particular we have that $\T_E^+ \dee \T_F^+$ in the category of operator spaces \cite{EleKak16}.
Therefore \cite[Theorem 5.10]{EleKak16} implies that $\T_E^+$ and $\T_F^+$ have strongly $\Delta$-equivalent TRO envelopes.
However TRO envelopes for operator algebras coincide with their C*-envelopes.
Since strong $\Delta$-equivalence for C*-algebras is Rieffel's strong Morita equivalence \cite[Corollary 5.2]{EleKak16}, then \cite{KatKri06} implies that $\O_E \sme \O_F$. 
\end{remark}

\section{Applications}

\subsection{Stable isomorphism}

Theorem \ref{T: sme} implies that the operator algebras are equivalent by the same TRO that gives $E \sme F$.
By \cite[Theorem 4.6]{EleKak16} if two operator spaces $X$ and $Y$ are strongly $\Delta$-equivalent by a $\si$-TRO then they are stably isomorphic.
Hence we get the following corollary.

\begin{corollary}\label{C: stable}
Let ${}_A E_A$ and ${}_B F_B$ be non-degenerate C*-correspon\-dences.
If they are strongly Morita equivalent by a $\si$-TRO, then their tensor algebras, Toeplitz-Pimsner and (relative) Cuntz-Pimsner algebras are stably isomorphic.
\end{corollary}

We provide a necessary and sufficient condition for $\sme$ to be induced by a $\si$-TRO.
Recall that $\K_\infty(E)$ coincides with the spatial tensor product of the compact operators $\K$ with $E$.
Following Lance \cite[Chapter 4]{Lan95}, it becomes a right Hilbert module over $\K_\infty(A)$ with the inner product given by the formula
\[
\sca{ [e_{ij}], [f_{ij}] }_{K_\infty (E)} = \left[ \sum_k \sca{e_{ki}, f_{kj}}_E \right]
\]
and the obvious right action.
It further becomes a C*-correspondence over $\K_\infty(A)$ with the left action given by
\[
\phi _{K_\infty (E)} ( [a_{ij}] ) [e_{ij}] = \left[ \sum_k \phi _E(a_{ik}) e_{kj} \right].
\]

\begin{proposition}\label{P: si-sme=stable}
Let ${}_A E_A$ and ${}_B F_B$ be non-degenerate C*-corresponden\-ces.
Then $E \sme F$ by a $\si$-TRO if and only if $\K_\infty(E) \simeq \K_\infty(F)$ as operator bimodules.
\end{proposition}

\begin{proof}
Suppose that $E \sme F$ by a $\si$-TRO $M$.
Then Corollary \ref{C: ns cond sme} implies that $E \dee F$ by a $\si$-TRO.
Therefore \cite[Corollary 4.7]{EleKak16} applies to give that $\K_\infty(E) \simeq \K_\infty(F)$ as operator spaces.
Since the strong $\Delta$-equivalence between $E$ and $F$ respects the operator bimodule structure we have that the isomorphism $\K_\infty(E) \simeq \K_\infty(F)$ extends to a bimodule isomorphism.

For the converse it suffices to show that $E \sme \K_\infty(E)$ by a $\si$-TRO.
For convenience suppose that ${}_A E_A$ is represented isometrically in a $\B(K)$.
Let $M = \bC \otimes \C$ acting on $K \otimes \bC$, where $\C$ denotes the column operators from $\bC$ to $\ell^2$.
Then we can check that
\[
\K_\infty(E) = \ncl{M E M^*} \qand E = \ncl{M^* \K_\infty(E) M}
\]
and that
\[
\K_\infty(A) = \ncl{M A M^*} \qand A = \ncl{M^* \K_\infty(A) M}.
\]
The proof is then completed by Corollary \ref{C: ns cond sme}.
\end{proof}

\begin{remark}
An alternative proof of Corollary \ref{C: stable} can be given through Proposition \ref{P: si-sme=stable}.
To this end one needs to check that there is a ``canonical'' isomorphism from $\T_{\K_\infty(E)}$ to $\K_\infty(\T(E))$ that fixes the tensor algebras.
This follows by an application of the Gauge-Invariant-Uniqueness-Theorem.
Hence by \cite{KatKri06} it also induces an isomorphism from $\O_{\K_\infty(E)}$ to $\K_\infty(\O_E)$.
The proof then follows by recalling that $\K_\infty(E) \simeq \K_\infty(F)$ as operator bimodules coincides with them being unitarily equivalent, due to Proposition \ref{P: uneq=cc}.
\end{remark}

\subsection{Aperiodic C*-correspondences}

We already discussed that $\dee$ is strictly stronger than $\sme$ of \cite{BMP00}.
As an application of our results we show that there is a large class of algebras where $\dee$ and $\sme$ coincide.
A C*-correspondence ${}_A E_A$ is \emph{aperiodic} in the sense of Muhly-Solel \cite[Definition 5.1]{MuhSol00} if for every $n \in \bZ_+$, for every $\xi \in E^{\otimes n}$ and every hereditary subalgebra $B$ of $A$ we have
\[
\inf\{\nor{\phi_n(a) \xi a} \mid a \geq 0, a \in B, \nor{a} = 1\} = 0.
\]
In their in-depth analysis, Muhly-Solel \cite[Theorem 7.2]{MuhSol00} show that if $E$ and $F$ are non-degenerate, injective and aperiodic then $E \sme F$ is equivalent to $\T_E^+ \sme \T_F^+$.
Injectivity is required only for the forward implication whereas the converse reads through also for non-injective C*-correspondences.
Therefore we obtain the following corollary.

\begin{corollary}\label{C: aperiodic}
Let $E$ and $F$ be non-degenerate and aperiodic.
The following are equivalent:
\begin{enumerate}
\item $E \sme F$ in the sense of Muhly-Solel \cite{MuhSol00};
\item $\T_E^+ \dee \T_F^+$ in the sense of Eleftherakis \cite{Ele14};
\item $\T_E^+ \sme \T_F^+$ in the sense of Blecher-Muhly-Paulsen \cite{BMP00}.
\end{enumerate}
In particular if $A$ and $B$ are $\si$-unital then every item above is equivalent to stable isomorphism of $\T_E^+$ with $\T_F^+$.
\end{corollary}

\begin{acknow}
The authors acknowledge support from the London Ma\-thematical Society.
George Eleftherakis visited Newcastle University under a Scheme 4 LMS grant (Ref: 41449).
Elias Katsoulis visited Newcastle University under a Scheme 4 LMS grant (Ref: 41424).
\end{acknow}


\end{document}